\documentclass{amsart}[12pt]
\usepackage{graphics}
\usepackage{amsmath,amsfonts,amssymb,amsthm,mathtools}

\numberwithin{equation}{section}

\usepackage[usesnames,svgnames]{xcolor}
\usepackage{tikz}
\usetikzlibrary{calc}
\usepackage[enableskew,vcentermath]{youngtab}
\usepackage{ytableau,mathdots,yhmath}
\usepackage[all]{xy}
\usepackage{enumerate}
\usepackage[colorinlistoftodos]{todonotes}
\usepackage{url}
\usepackage{hyperref}
\usepackage{verbatim}

\RequirePackage{cleveref}
\usepackage{hypcap}
\hypersetup{colorlinks=true, citecolor=darkblue, linkcolor=darkblue}
\definecolor{darkblue}{rgb}{0.0,0,0.7}

\newtheorem{theorem}{Theorem}[section]

\newtheorem{proposition}[theorem]{Proposition}

\newtheorem{corollary}[theorem]{Corollary}

\theoremstyle{definition}

\newtheorem{example}[theorem]{Example}
\newtheorem{remark}[theorem]{Remark}

\newtheorem*{proposition*}{Proposition}

\newcommand{\cD}{\ensuremath{\mathcal{D}}}
\newcommand{\cI}{\ensuremath{\mathcal{I}}}
\newcommand{\cM}{\ensuremath{\mathcal{M}}}

\newcommand{\Ifpf}{\ensuremath{\mathcal{I}^{FPF}}}

\newcommand{\id}{\ensuremath{\mathrm{id}}}

\newcommand{\fkS}{\ensuremath{\mathfrak{S}}}

\newcommand{\cyc}{\ensuremath{\mathrm{Cyc}}}
\newcommand{\ellhat}{\ensuremath{\hat{\ell}}}
\newcommand{\ellfpf}{\ensuremath{\hat{\ell}^{FPF}}}

\newcommand{\inv}{\ensuremath{\mathrm{Inv}}}
\newcommand{\invhat}{\ensuremath{\widehat{\inv}}}

\begin{document}


\title{Involutions under Bruhat order and labeled Motzkin Paths}  

\author[M. Coopman]{Michael Coopman}
\address[MC]{Department of Mathematics, University of Florida, Gainesville, FL 32601}
\email{m.coopman@ufl.edu}

\author[Z. Hamaker]{Zachary Hamaker}
\address[ZH]{Department of Mathematics, University of Florida, Gainesville, FL 32601}
\email{zhamaker@ufl.edu}

\date{\today}


\begin{abstract}
In this note, we introduce a statistic on Motzkin paths that describes the rank generating function of Bruhat order for involutions. 
Our proof relies on a bijection introduced by P.~Biane from permutations to certain labeled Motzkin paths and a recently introduced interpretation of this rank generating function in terms of visible inversions.
By restricting our identity to fixed-point-free (FPF) involutions, we recover an identity due to L.~Billera, L.~Levine and K.~M\'esz\'aros with a previous bijective proof by M.~Watson.
 Our work sheds new light on the Ethiopian dinner game.

\end{abstract}

\maketitle

\section{Introduction}
\label{s:introduction}
Let $\fkS_n$ be the symmetric group on $n$ elements, $\cI_n$ be the subset of involutions and $\Ifpf_n$ ($n$ even) be the subset of fixed-point-free (FPF) involutions.
Bruhat order on the symmetric group $(\mathfrak{S}_n,\leq)$ is a graded poset whose rank function counts the number of inversions.
The restriction of Bruhat order to $\cI$ and $\Ifpf$, first considered by RW~Richardson and T.~A.~Springer due to their relation with $K$--orbit closures~\cite{richardson1990bruhat}, are also graded partial orders (see~\cite{deodhar2001statistic,incitti2004bruhat}).
Let $R_{\cI_n}(q)$ and $R_{\Ifpf_{2n}}(q)$ be the rank generating functions of $(\cI_n,\leq)$ and $(\Ifpf_{2n},\leq)$, respectively.
For $\mathcal{M}_n$ the set of Motzkin paths of length $n$ and $\mu \in \mathcal{M}_n$, we introduce a generating function $H[\mu;q]$ in Equation~\eqref{eq:stat} satisfying the identity:
\begin{theorem} \label{t:main}
For all $n \in \mathbb{N},$
\[\sum_{\mu \in \mathcal{M}_n} H[\mu;q] = R_{\cI_n}(q),\]
\end{theorem}

Our proof relies on a bijection due to P.~Biane~\cite{biane1993permutations} that, as observed in~\cite{barnabei2011restricted}, maps involutions to Motzkin paths with labeled down steps.
In~\cite{hamaker2019schur}, Z.~Hamaker, E.~Marberg and B.~Pawlowski introduce \emph{visible inversions} as a combinatorial interpretation of rank in $(\cI_n,\leq)$.
We show $H[\mu;q]$ counts visible inversions for the involutions corresponding to $\mu$.
Similarly, Biane's bijection maps $\Ifpf_{2n}$ to $\cD_{2n}$, the set of Dyck paths with length $2n$ and labeled down steps.
As a consequence, we recover the following identity:
\begin{corollary}
	\label{c:watson}
	For all $n \in \mathbb{N}$,
	\[
	\sum_{\delta \in \cD_{2n}} H[\delta;q] = q^n R_{\Ifpf_{2n}}(q) = q^n \prod_{k =1}^n [2k{-}1]_q.
	\]
\end{corollary}

An equivalent form of Corollary~\ref{c:watson} (see Equation~\eqref{eq:BLM}) is~\cite[Cor.~8]{billera2015decompose}, where a discussion of related results appears.
M. Watson gives a bijective proof~\cite{watson2014bruhat}, but his argument is more involved than ours since it proves a stronger statement.
He introduces a partial order on full rook placements for certain diagrams and shows it is isomorphic to Bruhat order on $\Ifpf$ using his bijection.
In fact, Watson's bijection is equivalent to Biane's when restricted to $\Ifpf$, and his approach can be extended to $\cI_n$ as explained in Section~\ref{ss:watson}.
In Section~\ref{ss:blm} we observe that the bijection in~\cite{billera2015decompose}, which is not used directly in their proof of Corollary~\ref{c:watson}, is also equivalent to Biane's when restricted to $\Ifpf_{2n}$.
Their proof arises from combinatorics related to the Ethiopian dinner game introduced in~\cite{levine2012make}, and we explain how to interpret Corollary~\ref{c:watson} in this context.

\smallskip
\smallskip

\noindent\textbf{Acknowledgements:} The authors thank Oliver Pechenik, Nathan Reading and Vic Reiner for helpful conversations and encouragement and Marilena Barnabei for suggestions improving our exposition.
This work was partially supported by NSF grant DMS-2054423.

\section{Combinatorial Structures}
\label{s:structures}
\subsection{Permutations, Bruhat order, and visible inversions}
Let $\mathfrak{S}_n$ be the set of permutations from $[n] = \{1, 2, \hdots, n\}$ to itself.
All elements of $\mathfrak{S}_n$ can be written as a product of disjoint cycles.
For $\pi \in \mathfrak{S}_n$, let $\text{Inv}(\pi) = \{(i,j) \in [n]^2 \mid i<j \text{ and } \pi(i) > \pi(j)\}$ be the set of \emph{inversions} of $\pi$ and $\ell(\pi) = |\inv(\pi)|$. 
\emph{Bruhat order} $(\mathfrak{S}_n,\leq)$ is a partial order defined as the transitive closure of the relations: $\pi \lessdot (i \; j) \pi$ if $\ell((i \; j)\pi) = \ell(\pi) + 1$ with $(i \; j) \in \mathfrak{S}_n$.
The rank function of Bruhat order is $\ell.$

A permutation $\pi$ is an \emph{involution} if $\pi^2 = \id_n,$ and is \emph{fixed-point-free} (FPF) if $\pi(i) \neq i$ for all $i\in [n]$.
Let $\cI_n$ be the set of involutions of size $n$ and $\Ifpf_{2n}$ be the set of FPF involutions of size $2n$.
Bruhat order induces partial orders on $\cI_n$ and $\Ifpf_{2n}$ (see Figure~\ref{fig:Bruhat}).
For $\sigma \in \cI_n$, define $\cyc(\sigma) = \{(i,j) \in [n] \times [n]: i < j = \sigma(i)\}$, $c(\sigma) = |\cyc(\sigma)|$ and
\[
\ellhat = \frac{\ell +c}{2} \quad \mbox{and} \quad \ellfpf = \ellhat - c = \frac{\ell - c}{2}, \quad \mbox{respectively}.
\]
	
\begin{proposition}[{\cite[Thm.~5.2]{incitti2004bruhat} and \cite[Thm.~1.3]{deodhar2001statistic}}]
\label{p:graded}
Both $(\cI_n,\leq)$ and $(\Ifpf_{2n},\leq)$ are graded posets with respective rank functions $\ellhat$ and $\ellfpf$.

\end{proposition}

Note that $\ellfpf$ is well-defined for $\sigma \in \cI_n$, where it is studied in~\cite{deodhar2001statistic}.
See Equation~\eqref{eq:deodhar} for further discussion.
When $\tau \in \Ifpf_{2n}$, we have $c(\tau) = n$ so $\ellfpf(\tau) = \ellhat(\tau) - n = \frac{\ell(\tau) - n}{2}$.

For $\pi \in \mathfrak{S}_n$, say an inversion $(i,j) \in \inv(\pi)$ is \emph{visible} if $\pi(j) \leq \min(i, \pi(i))$.
Let $\invhat(\pi)$ be the set of visible inversions for $\pi$, respectively.
For example, with $\sigma = (1,5)(2)(3,6)(4) = 526413$, we have
\[
 \invhat(\sigma) = \{(1,5),(2,5),(3,5),(3,6),(4,5),(4,6)\} \quad \mbox{while} \quad
\inv(\sigma) = \invhat(\sigma) \cup  \{(1,2),(1,4),(1,6),(3,4)\}.
\]
Since $c(\sigma) = 2$, we see $\ellhat(\sigma) = \frac{10+2}{2} = 6 = |\invhat(\sigma)|$.
Note $\cyc(\sigma) \subseteq \invhat(\sigma)$ for all $\sigma \in \cI_n$.
\begin{proposition}[{\cite[Lem.~4.11]{hamaker2019schur}}]
	For $\sigma \in \cI_n$ we have 
	\[
	\ellhat(\sigma) = |\invhat(\sigma)| \quad \mbox{hence}  \quad \ellfpf(\sigma) = |\invhat(\sigma) \setminus \cyc(\sigma)|.
	\]
\end{proposition}

For $(P,\leq)$ a graded poset with rank function $r$, let $R_P(q) = \sum_{p \in P} q^{r(p)}$ be the \emph{rank generating function} of $P$.
By Proposition~\ref{p:graded}, we have
\begin{equation}
	\label{eq:rank-functions}
R_{\cI_n}(q) = \sum_{\sigma \in \cI_n} q^{\ellhat(\sigma)} \quad \mbox{and} \quad R_{\Ifpf_{2n}}(q) = \sum_{\tau \in \Ifpf_{2n}} q^{\ellfpf(\sigma)}.
\end{equation}
For $n \geq 1$, let $[n]_q = \frac{1-q^n}{1-q}$.
Using induction, it is easy to show that $R_{\Ifpf_{2n}}(q) = \prod^n_{k=1} [2k-1]_q := [2n-1]_q!!$, but $R_{\cI_n}(q)$ does not have a simple closed form.
However, it can be efficiently computed using the recurrence
\[R_{\cI_n}(q) = R_{\cI_{n-1}}(q) + q[n-1]_q R_{\cI_{n-2}}(q).
\]

\begin{figure}
\begin{center}
    \begin{tikzpicture} [scale = 0.75]
        \coordinate (A) at (0,0);
        \node at (A) {(1)(2)(3)(4)};
        \coordinate (B) at (-4,2);
        \node at (B) {(12)(3)(4)};
        \coordinate (C) at (0,2);
        \node at (C) {(1)(23)(4)};
        \coordinate (D) at (4,2);
        \node at (D) {(1)(2)(34)};
        \coordinate (E) at (-4,4);
        \node at (E) {(13)(2)(4)};
        \coordinate (F) at (0,4);
        \node at (F) {(12)(34)};
        \coordinate (G) at (4,4);
        \node at (G) {(1)(24)(3)};
        \coordinate (H) at (-2,6);
        \node at (H) {(14)(2)(3)};
        \coordinate (I) at (2,6);
        \node at (I) {(13)(24)};
        \coordinate (J) at (0,8);
        \node at (J) {(14)(23)};
        
        \draw[thick, color = black] ($(A)+(-0.4,0.4)$) -- ($(B)-(-0.4,0.4)$);
        \draw[thick, color = black] ($(A)+(0,0.4)$) -- ($(C)-(0,0.4)$);
        \draw[thick, color = black] ($(A)+(0.4,0.4)$) -- ($(D)-(0.4,0.4)$);
        \draw[thick, color = black] ($(B)+(0,0.4)$) -- ($(E)-(0,0.4)$);
        \draw[thick, color = black] ($(B)+(0.4,0.4)$) -- ($(F)-(0.4,0.4)$);
        \draw[thick, color = black] ($(C)+(-0.4,0.4)$) -- ($(E)-(-0.4,0.4)$);
        \draw[thick, color = black] ($(C)+(0.4,0.4)$) -- ($(G)-(0.4,0.4)$);
        \draw[thick, color = black] ($(D)+(-0.4,0.4)$) -- ($(F)-(-0.4,0.4)$);
        \draw[thick, color = black] ($(D)+(0,0.4)$) -- ($(G)-(0,0.4)$);
        \draw[thick, color = black] ($(E)+(0.2,0.4)$) -- ($(H)-(0.2,0.4)$);
        \draw[thick, color = black,dashed] ($(E)+(0.6,0.4)$) -- ($(I)-(0.6,0.4)$);
        \draw[thick, color = black,dashed] ($(F)+(-0.2,0.4)$) -- ($(H)-(-0.2,0.4)$);
        \draw[thick, color = black] ($(F)+(0.2,0.4)$) -- ($(I)-(0.2,0.4)$);
        \draw[thick, color = black] ($(G)+(-0.6,0.4)$) -- ($(H)-(-0.6,0.4)$);
        \draw[thick, color = black,dashed] ($(G)+(-0.2,0.4)$) -- ($(I)-(-0.2,0.4)$);
        \draw[thick, color = black] ($(H)+(0.2,0.4)$) -- ($(J)-(0.2,0.4)$);
        \draw[thick, color = black] ($(I)+(-0.2,0.4)$) -- ($(J)-(-0.2,0.4)$);
    \end{tikzpicture}
    \end{center}
\caption{Bruhat order for $\cI_4$.
The dashed edges do not appear in $<_W$ (see Remark~\ref{rem:poset}).}
\label{fig:Bruhat}	
\end{figure}
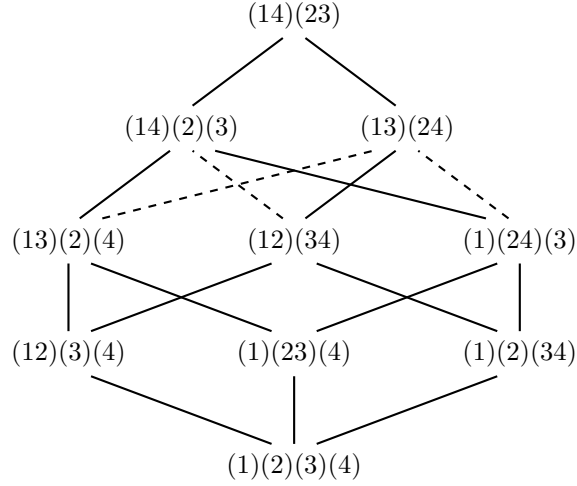

\subsection{Labeled Motzkin path, statistics and bijections}

Recall a \emph{Motzkin path} of \emph{length} $n$ is a lattice path from (0,0) to (n,0) consisting of up steps $U = (1,1)$, down steps $D = (-1,1)$, and horizontal steps $L = (0,1)$  that does not go below the x-axis.
Let $\cM_n$ be the set of Motzkin paths of length $n$.
For $\mu \in \cM_n$, the \emph{height} of the $i$th step is the larger of the step's $y$-coordinates, denoted $h_i(\mu)$.

For our purposes, a \emph{labeled Motzkin path} is a Motzkin path $\mu = \mu_1 \dots \mu_n$ where each down step $\mu_i = D$ is labeled with an integer $\lambda_i \in [h_i(\mu)]$.
Let $\cM^L_n$ be the set of labeled Motzkin paths of length $n$.
As observed in~\cite{barnabei2011restricted}, the bijection in~\cite{biane1993permutations} restricts to a simple bijection $\phi:\cI_n \to \cM^L_n$.
For $\sigma \in \cI_n$, let $\mu = \mu_1 \dots \mu_n$ with
\[
\mu_i = \begin{cases}
 U & i < \sigma(i)\\
 L & i = \sigma(i)\\
 D & i > \sigma(i)
 \end{cases}.
\]
When $\mu_i = D$, let $\lambda_i$ be the number of integers $j \geq i$ such that $\pi(j) \leq \pi(i)$ and define $\phi(\sigma) = (\mu,\lambda)$.
Note that this description deviates from Biane's, in that his Motzkin paths receive two labels at every horizontal and down step.
For involutions, the double labels are redundant: down steps are labeled twice by the same value and horizontal steps are labeled twice by their heights.

For $(\mu,\lambda) \in \cM^L_n$ with $\mu = \mu_1 \dots \mu_n$, define
\[
H_i(\mu,\lambda) = \begin{cases}
 \lambda_i(\mu)-1 & \mu_i = D\\
  h_i(\mu) & 	\mu_i \neq D
 \end{cases}
\quad \mbox{and} \quad H(\mu,\lambda) = \sum_{i = 1}^n H_i(\mu,\lambda).
\]
Next, define
\begin{equation}
	\label{eq:stat}
	H[\mu;q] = \sum_{\lambda:(\mu,\lambda) \in \cM_n} q^{H(\mu,\lambda)}.
\end{equation}
Alternatively, with 
\begin{equation}
H_i[\mu;q] = \begin{cases} [h_i(\mu)]_q & \mu_i = D\\
 q^{h_i(\mu)}	& \mbox{else}
 \end{cases},
 \quad \mbox{we have} \quad H[\mu;q] = \prod_{i=1}^n H_i[\mu;q].
\label{eq:H-stat}	
\end{equation}
Note $H[\mu;1]$ counts the number of labelings $\lambda$ for $\mu$.
Define $\tilde{H}_i[\mu;q]$ and $\tilde{H}[\mu;q]$ analogously, omitting a factor of $q$ on each upstep $\mu_i = U$.
For example, with $(\mu,\lambda)$ as in Figure~\ref{fig:bijection}, we have $H(\mu,\lambda) = 18$ and
\[
\{H_i[\mu;q]\}_{i=1}^{11} = \{q, q^2, q^2,1+q, q^2, q^3, q^3, q^3, 1+q+q^2, 1+q, 1\},
\]
so $H[\mu;q] = q^{16}(1+q)^2(1+q+q^2)$ and $\tilde{H}[\mu;q] = q^{12}(1+q)^2(1+q+q^2)$.

A \emph{Dyck path} is a Motzkin path without any horizontal steps.
Necessarily, Dyck paths are always of even length.
Let $\cD_{2n}$ be the set of Dyck paths of length $2n$.
By restricting $\phi$ to $\Ifpf_{2n}$, we obtain a bijection between $\Ifpf_{2n}$ and labeled Dyck paths of length $2n$, which are also known as \emph{Hermite histories}.
This restriction is well known when viewing $\Ifpf_{2n}$ as the set of perfect matchings of the complete graph $K_{2n}$. 
Since Dyck paths of length $2n$ always have $n$ down steps, we have $H[\mu;q] = q^n\tilde{H}[\mu;q]$.

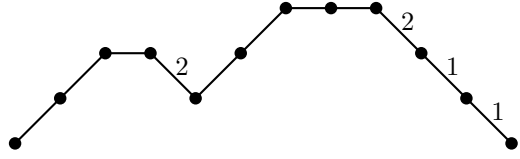
\begin{figure}
    \begin{center}
    \begin{tikzpicture} [scale = 0.6]
        \node[circle, fill = black, scale = 0.5] at (0,0) {};
        \node[circle, fill = black, scale = 0.5] at (1,1) {};
        \node[circle, fill = black, scale = 0.5] at (2,2) {};
        \node[circle, fill = black, scale = 0.5] at (3,2) {};
        \node[circle, fill = black, scale = 0.5] at (4,1) {};
        \node[circle, fill = black, scale = 0.5] at (5,2) {};
        \node[circle, fill = black, scale = 0.5] at (6,3) {};
        \node[circle, fill = black, scale = 0.5] at (7,3) {};
        \node[circle, fill = black, scale = 0.5] at (8,3) {};
        \node[circle, fill = black, scale = 0.5] at (9,2) {};
        \node[circle, fill = black, scale = 0.5] at (10,1) {};
        \node[circle, fill = black, scale = 0.5] at (11,0) {};
        \draw[thick, color = black] (0,0) -- (2,2) -- (3,2) -- (4,1) -- (6,3) -- (8,3) -- (11,0);
        \node at (3.7,1.7) {2};
        \node at (8.7,2.7) {2};
        \node at (9.7,1.7) {1};
        \node at (10.7,0.7) {1};
    \end{tikzpicture}
    \end{center}
\caption{The labeled Motzkin path corresponding to $\sigma=(1 \; 10)(2 \; 4)(3)(5 \; 9)(6 \; 11)(7)(8)$.}
\label{fig:bijection}	
\end{figure}

\section{Proof of Main Results}
\label{s:main}
\begin{proof}[Proof of Theorem \ref{t:main}]
Let $\sigma \in \cI_n$ and $(\mu,\lambda) = \phi(\sigma)$.
We split $\invhat(\sigma)$ into 2 disjoint classes $S_1$ and $S_2$.

\textbf{Class 1: } $S_1 = \{ (i,j) \in \invhat(\sigma) \mid i \leq \sigma(i)\}$.\\ 
Fix $s \in [n]$ such that $s \leq \sigma(s)$.
The number of integers $t$ such that $(s,t) \in S_1$ is equal to the number of $(u, t) \in \cyc(\sigma)$ such that $u \leq s < t$.
As $s$ corresponds to a horizontal or up step in $\mu$, $h_s(\mu)$ counts the number of up steps in excess of the number of down step among the first $s$ steps.
Through the bijection $\phi$, $h_s(\mu)$ counts the number of $(u, \sigma(u)) \in \cyc(\sigma)$ such that $u \leq s < \sigma(u)$.
Thus, $|S_1|$ is equal to the sum of the heights of all horizontal and up steps of $\mu$.

\textbf{Class 2: }: $S_2 = \{ (i,j) \in \invhat(\sigma) \mid i > \sigma(i)\}$.\\ 
As before, fix $s \in [n]$ such that $s > \sigma(s)$ and count the number of integers $t$ such that $(s,t) \in S_2$.
This requires counting the number of $(u, t) \in \cyc(\sigma)$ such that $u \leq \sigma(s) < s < t$.
Note that the weak inequality is changed to strict as $s < t$.
As $s$ corresponds to an down step in $\mu$, the $s$-th step is labeled by some integer $1 \leq \lambda_s \leq h_s(\mu)$.
Note that $\lambda_s$ refers to the number of integers $t \geq s$ such that $\sigma(t) \leq \sigma(s)$.
Thus, $\lambda_s -1$ refers to the number of $(\sigma(t), t) \in \cyc(\sigma)$ such that $\sigma(t) < \sigma(s) < s < t$.
Thus, $|S_2|$ is equal to the sum of the labels (minus 1) of all down steps of $\mu$. \\

As a consequence, we see $|\invhat(\sigma)| = H(\mu,\lambda)$.
Therefore since $\ellhat(\sigma) = |\invhat(\sigma)|$ we have
\[
    R_{\mathcal{I}_n}(q) = \sum_{\mu \in \mathcal{M}_n} H[\mu; q],
    \vspace{ -1 pt}
\]
which completes our proof.
\end{proof}

\begin{example}
In Figure ~\ref{fig:bijection}, $\sigma$ has eighteen visible inversions.
The two from Class 2 are $(4,10)$ and $(9,10)$, corresponding to the labels $2$ in positions $4$ and $9$, respectively.
The remaining sixteen visible inversions are from Class 1, and are encoded in the underlying Motzkin path.
\end{example}

One could also derive Theorem~\ref{t:main} from~\cite[\S3.2]{biane1993permutations} and Proposition~\ref{p:graded}, but this would require a non-trivial modification of Biane's statistics.
We prefer a self-contained proof, given its ease and brevity.

\begin{proof} [Proof of Corollary ~\ref{c:watson}]
    The restriction of $\phi$ to $\Ifpf_{2n}$ maps FPF involutions to labeled Dyck paths of length $2n$.
    For $\tau \in \Ifpf_{2n},$ note that $\ellhat(\tau) - \ellfpf(\tau) = \frac{1}{2}(|\cyc(\tau)| + n) = n$.
    Thus, weighting $\Ifpf$ by $\ellhat$ rather than $\ellfpf$ results in $q^n R_{\Ifpf_{2n}}(q)$ as desired. 
\end{proof}

If we instead sum over the statistic $\tilde{H}[\mu;q]$ defined after Equation~\eqref{eq:H-stat}, we obtain
\begin{equation}
\label{eq:BLM}
\sum_{\delta \in \cD_{2n}} \tilde{H}[\delta;q] = R_{\Ifpf_{2n}}(q) = \prod_{k=1}^n [2k-1]_q.	
\end{equation}

\begin{remark}
\label{rem:poset}
An alternate approach to proving Theorem~\ref{t:main}, used in~\cite{watson2014bruhat} to prove Corollary~\ref{c:watson}, is to describe the image of $(\cI_n,\leq)$ under $\phi$.
One may show for an arbitrary cover relation $\tau \lessdot \sigma$ that $H(\phi(\sigma)) = H(\phi(\tau)) + 1$.
An explicit description of cover relations appears in~\cite[Tab.~1]{incitti2004bruhat}.
This strategy is easier to implement for the \emph{weak order for involutions}, denoted $(\cI_n,\leq_W)$, another partial order introduced in~\cite{richardson1990bruhat} with the same rank generating function $R_{\cI_n}(q)$ but fewer cover relations.
Using Figure~\ref{fig:weak}, one can verify for $\tau \lessdot_W \sigma$ that $H(\phi(\sigma)) = H(\phi(\tau)) + 1$ .

\end{remark}

\section{Connections to related work}

\subsection{Watson's description} \label{ss:watson}
We explain the relationship between Corollary ~\ref{c:watson} and~\cite[Thm. 1]{watson2014bruhat}:\begin{equation}
\sum_{\delta \in \cD_{2n}} \prod_{i \in [n]} q^{d_i(\delta) - 2i} [d_i(\delta) - 2i + 1]_q = \prod_{k=1}^n [2k-1]_q., 
\label{eq:watson}	
\end{equation}
where $d_i(\delta)$ is the position of the $i$-th down step in $\delta$.
If $d_i(\delta) = i',$ then the height of the $i'$-th step would the number of up steps $(d_i(\delta) - i)$ before $i'$ in excess of the number of down step $(i - 1)$ before $i'$.
Thus, we have that the height of the $i$-th down step is $d_i(\delta) - 2i + 1$, so Equations~\eqref{eq:BLM} and~\eqref{eq:watson} are equivalent.

Watson interprets Equation~\eqref{eq:watson} in terms of fully packed rook placements on Young diagrams.
These correspond to labeled Dyck paths as follows.
For $(\delta = \delta_1 \dots \delta_{2n},\lambda)$ a labeled Dyck path, the corresponding Young diagram is cut out by the lattice path from $(0,n)$ to $(n,0)$ whose $i$-th step is vertical if $\delta_i = D$ and horizontal if $\delta_i = U$.
Each downward step is associated with a row on the diagram.
Starting with the first downward step, place a rook in the $\lambda$-th leftmost valid (not in the same column as another rook) spot in the associated row.
To show this is reversible, note that $\delta$ can be recovered from the shape of the diagram and the down step labeling can be iteratively recovered by the rooks.

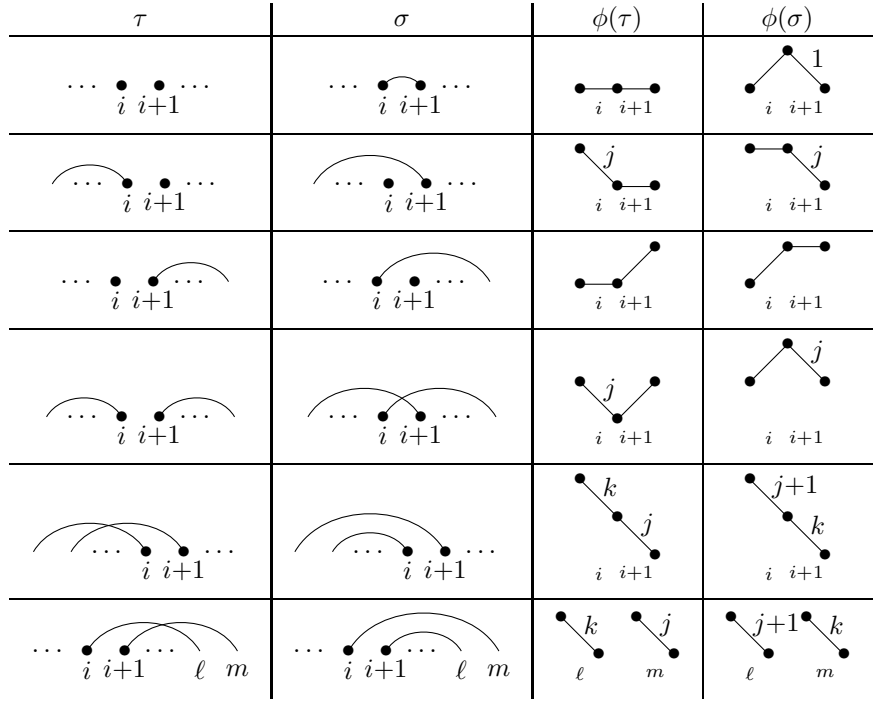
\begin{figure}
\label{fig:weak}	
\begin{tabular}{c|c|c|c}
$\tau$ & $\sigma$ & $\phi(\tau)$ & $\phi(\sigma)$\\
\hline
\begin{tikzpicture}[scale=.5]
	\node at (-1,0) {$\dots$};
	\node at (2,0) {$\dots$};
	
	\node at (0,0) (i) {$\bullet$};
	\node at (1,0) (i+1) {$\bullet$};
	
	\node at (0,-.5) {$i$};
	\node at (1,-.5) {$i{+}1$};
\end{tikzpicture}
&
\begin{tikzpicture}[scale=.5]
	\node at (-1,0) {$\dots$};
	\node at (2,0) {$\dots$};

	\node at (0,0) (i) {$\bullet$};
	\node at (1,0) (i+1) {$\bullet$};
	
	\node at (0,-.5) {$i$};
	\node at (1,-.5) {$i{+}1$};
	
	\draw[bend right] (i.center) to [out=60,in=120] (i+1.center);
\end{tikzpicture}
&
\begin{tikzpicture}[scale=.5]
\node at (0,0) (0) {$\bullet$};
\node at (1,0) (1) {$\bullet$};
\node at (2,0) (2) {$\bullet$};
	\draw (0.center) -- (1.center) -- (2.center);
	
\node at (.5,-.5) {\scriptsize$i$};
\node at (1.5,-.5) {\scriptsize$i{+}1$};
\end{tikzpicture}
&
\begin{tikzpicture}[scale=.5]
\node at (0,0) (0) {$\bullet$};
\node at (1,1) (1) {$\bullet$};
\node at (2,0) (2) {$\bullet$};
\draw (0.center) -- (1.center) --  (2.center);

\node at (1.8,.8) {1};

\node at (.5,-.5) {\scriptsize$i$};
\node at (1.5,-.5) {\scriptsize$i{+}1$};
\end{tikzpicture}
\\
\hline
\begin{tikzpicture}[scale=.5]
	\node at (-1,0) {$\dots$};
	\node at (2,0) {$\dots$};
	
	\node at (0,0) (i) {$\bullet$};
	\node at (1,0) (i+1) {$\bullet$};
	
	\node at (0,-.5) {$i$};
	\node at (1,-.5) {$i{+}1$};
	
	\draw[bend right] (-2,0) to [out=60,in=120] (i.center);
\end{tikzpicture}
&
\begin{tikzpicture}[scale=.5]
	\node at (-1,0) {$\dots$};
	\node at (2,0) {$\dots$};

	\node at (0,0) (i) {$\bullet$};
	\node at (1,0) (i+1) {$\bullet$};
	
	\node at (0,-.5) {$i$};
	\node at (1,-.5) {$i{+}1$};
	
	\draw[bend right] (-2,0) to [out=60,in=120] (i+1.center);
\end{tikzpicture}
&
\begin{tikzpicture}[scale=.5]
\node at (0,1) (0) {$\bullet$};
\node at (1,0) (1) {$\bullet$};
\node at (2,0) (2) {$\bullet$};
	\draw (0.center) -- (1.center) -- (2.center);
\node at (.8,.8) {$j$};
\node at (.5,-.5) {\scriptsize$i$};
\node at (1.5,-.5) {\scriptsize$i{+}1$};
\end{tikzpicture}
&
\begin{tikzpicture}[scale=.5]
\node at (0,1) (0) {$\bullet$};
\node at (1,1) (1) {$\bullet$};
\node at (2,0) (2) {$\bullet$};
\draw (0.center) -- (1.center) --  (2.center);

\node at (1.8,.8) {$j$};
\node at (.5,-.5) {\scriptsize$i$};
\node at (1.5,-.5) {\scriptsize$i{+}1$};
\end{tikzpicture}
\\
\hline
\begin{tikzpicture}[scale=.5]
	\node at (-1,0) {$\dots$};
	\node at (2,0) {$\dots$};
	
	\node at (0,0) (i) {$\bullet$};
	\node at (1,0) (i+1) {$\bullet$};
	
	\node at (0,-.5) {$i$};
	\node at (1,-.5) {$i{+}1$};
	
	\draw[bend left] (3,0) to [out=-60,in=-120] (i+1.center);
\end{tikzpicture}
&
\begin{tikzpicture}[scale=.5]
	\node at (-1,0) {$\dots$};
	\node at (2,0) {$\dots$};

	\node at (0,0) (i) {$\bullet$};
	\node at (1,0) (i+1) {$\bullet$};
	
	\node at (0,-.5) {$i$};
	\node at (1,-.5) {$i{+}1$};
	
	\draw[bend left] (3,0) to [out=-60,in=-120] (i.center);
\end{tikzpicture}
&
\begin{tikzpicture}[scale=.5]
\node at (0,0) (0) {$\bullet$};
\node at (1,0) (1) {$\bullet$};
\node at (2,1) (2) {$\bullet$};
	\draw (0.center) -- (1.center) -- (2.center);
	\node at (.5,-.5) {\scriptsize$i$};
\node at (1.5,-.5) {\scriptsize$i{+}1$};
\end{tikzpicture}
&
\begin{tikzpicture}[scale=.5]
\node at (0,0) (0) {$\bullet$};
\node at (1,1) (1) {$\bullet$};
\node at (2,1) (2) {$\bullet$};
\draw (0.center) -- (1.center) --  (2.center);
\node at (.5,-.5) {\scriptsize$i$};
\node at (1.5,-.5) {\scriptsize$i{+}1$};
\end{tikzpicture}
\\
\hline
\begin{tikzpicture}[scale=.5]
	\node at (-1,0) {$\dots$};
	\node at (2,0) {$\dots$};
	
	\node at (0,0) (i) {$\bullet$};
	\node at (1,0) (i+1) {$\bullet$};
	
	\node at (0,-.5) {$i$};
	\node at (1,-.5) {$i{+}1$};
	
	\draw[bend right] (-2,0) to [out=60,in=120] (i.center);
	\draw[bend right] (3,0) to [out=-60,in=-120] (i+1.center);
\end{tikzpicture}
&
\begin{tikzpicture}[scale=.5]
	\node at (-1,0) {$\dots$};
	\node at (2,0) {$\dots$};

	\node at (0,0) (i) {$\bullet$};
	\node at (1,0) (i+1) {$\bullet$};
	
	\node at (0,-.5) {$i$};
	\node at (1,-.5) {$i{+}1$};
	
	\draw[bend right] (-2,0) to [out=60,in=120] (i+1.center);
	\draw[bend right] (3,0) to [out=-60,in=-120] (i.center);
\end{tikzpicture}
&
\begin{tikzpicture}[scale=.5]
\node at (0,1) (0) {$\bullet$};
\node at (1,0) (1) {$\bullet$};
\node at (2,1) (2) {$\bullet$};
	\draw (0.center) -- (1.center) -- (2.center);
\node at (.8,.8) {$j$};
\node at (.5,-.5) {\scriptsize$i$};
\node at (1.5,-.5) {\scriptsize$i{+}1$};
\end{tikzpicture}
&
\begin{tikzpicture}[scale=.5]
\node at (0,1) (0) {$\bullet$};
\node at (1,2) (1) {$\bullet$};
\node at (2,1) (2) {$\bullet$};
\draw (0.center) -- (1.center) --  (2.center);

\node at (1.8,1.8) {$j$};
\node at (.5,-.5) {\scriptsize$i$};
\node at (1.5,-.5) {\scriptsize$i{+}1$};
\end{tikzpicture}

\\
\hline
\begin{tikzpicture}[scale=.5]
	\node at (-1,0) {$\dots$};
	\node at (2,0) {$\dots$};
	
	\node at (0,0) (i) {$\bullet$};
	\node at (1,0) (i+1) {$\bullet$};
	
	\node at (0,-.5) {$i$};
	\node at (1,-.5) {$i{+}1$};
	
	\draw[bend right] (-2,0) to [out=60,in=120] (i+1.center);
	\draw[bend right] (-3,0) to [out=60,in=120] (i.center);
\end{tikzpicture}
&
\begin{tikzpicture}[scale=.5]
	\node at (-1,0) {$\dots$};
	\node at (2,0) {$\dots$};

	\node at (0,0) (i) {$\bullet$};
	\node at (1,0) (i+1) {$\bullet$};
	
	\node at (0,-.5) {$i$};
	\node at (1,-.5) {$i{+}1$};
	
	\draw[bend right] (-2,0) to [out=60,in=120] (i.center);
	\draw[bend right] (-3,0) to [out=60,in=120] (i+1.center);
\end{tikzpicture}
&
\begin{tikzpicture}[scale=.5]
\node at (0,2) (0) {$\bullet$};
\node at (1,1) (1) {$\bullet$};
\node at (2,0) (2) {$\bullet$};
	\draw (0.center) -- (1.center) -- (2.center);
\node at (.8,1.8) {$k$};
\node at (1.8,.8) {$j$};
\node at (.5,-.5) {\scriptsize$i$};
\node at (1.5,-.5) {\scriptsize$i{+}1$};
\end{tikzpicture}
&
\begin{tikzpicture}[scale=.5]
\node at (0,2) (0) {$\bullet$};
\node at (1,1) (1) {$\bullet$};
\node at (2,0) (2) {$\bullet$};
\draw (0.center) -- (1.center) --  (2.center);
\node at (1.2,1.8) {$j{+}1$};
\node at (1.8,.8) {$k$};
\node at (.5,-.5) {\scriptsize$i$};
\node at (1.5,-.5) {\scriptsize$i{+}1$};
\end{tikzpicture}

\\
\hline
\begin{tikzpicture}[scale=.5]
	\node at (-1,0) {$\dots$};
	\node at (2,0) {$\dots$};
	
	\node at (0,0) (i) {$\bullet$};
	\node at (1,0) (i+1) {$\bullet$};
	
	\node at (0,-.5) {$i$};
	\node at (1,-.5) {$i{+}1$};
	
	\node at (3,-.5) {$\ell$};
	\node at (4,-.5) {$m$};
	
	\draw[bend right] (3,0) to [out=-60,in=-120] (i.center);
	\draw[bend right] (4,0) to [out=-60,in=-120] (i+1.center);
\end{tikzpicture}
&
\begin{tikzpicture}[scale=.5]
	\node at (-1,0) {$\dots$};
	\node at (2,0) {$\dots$};

	\node at (0,0) (i) {$\bullet$};
	\node at (1,0) (i+1) {$\bullet$};
	
	\node at (0,-.5) {$i$};
	\node at (1,-.5) {$i{+}1$};
	
	\node at (3,-.5) {$\ell$};
	\node at (4,-.5) {$m$};
	
	\draw[bend right] (3,0) to [out=-60,in=-120] (i+1.center);
	\draw[bend right] (4,0) to [out=-60,in=-120] (i.center);
\end{tikzpicture}
&
\begin{tikzpicture}[scale=.5]
\node at (0,1) (0) {$\bullet$};
\node at (1,0) (1) {$\bullet$};
\node at (2,1) (2) {$\bullet$};
\node at (3,0) (3) {$\bullet$};
	\draw (0.center) -- (1.center);
	\draw (2.center) -- (3.center);
\node at (.8,.8) {$k$};
\node at (2.8,.8) {$j$};

\node at (.5,-.5) {\scriptsize$\ell$};
\node at (2.5,-.5) {\scriptsize$m$};
\end{tikzpicture}
&
\begin{tikzpicture}[scale=.5]
\node at (0,1) (0) {$\bullet$};
\node at (1,0) (1) {$\bullet$};
\node at (2,1) (2) {$\bullet$};
\node at (3,0) (3) {$\bullet$};
	\draw (0.center) -- (1.center);
	\draw (2.center) -- (3.center);
\node at (1.2,.8) {$j{+}1$};
\node at (2.8,.8) {$k$};

\node at (.5,-.5) {\scriptsize$\ell$};
\node at (2.5,-.5) {\scriptsize$m$};
\end{tikzpicture}

\end{tabular}
\caption{Cover relations $\tau \lessdot_W \sigma$ for $\tau,\sigma \in \cI_n$, with involutions depicted as partial matchings.}	

\end{figure}
\subsection{$\Ifpf_{2n}$ and the Ethiopian dinner game} \label{ss:blm}
In~\cite{billera2015decompose}, the authors study the \emph{Ethiopian dinner game}.
Alice and Bob are sharing a meal with morsels $\{1,\dots,2n\}$.
Bob prefers larger-valued morsels, while Alice prefers larger values $\pi_k$ for some permutation $\pi = \pi_1 \dots \pi_{2n}$.
The players alternate choice of morsel, beginning with Alice.
The optimal strategy for both players is best explained by describing the reverse order in which morsels are chosen -- Bob chooses the smallest unselected morsel in $\pi_1 \dots \pi_{2n}$, then Alice chooses the leftmost unselected morsel and so on, resulting in an allocation function $w:[2n] \to \{A,B\}$~\cite{kohler1971class,levine2012make}.

The main result~\cite[Thm.~1]{billera2015decompose} is a bijection from permutations in $S_{2n}$ to pairs of labeled Dyck paths obtained by analyzing the Ethiopian dinner game.
Given a permutation $\pi = \pi_1 \dots \pi_{2n}$ with optimal allocation $w$, construct the Dyck paths $\delta^A(\pi) = \delta_1^A \dots \delta_{2n}^A$ and $\delta^B(\pi) = \delta^B_1 \dots \delta^B_{2n+2}$ by setting
\[
\delta_i^A = \begin{cases}
	D & w(\pi_i) = A\\
	U & \mbox{else},
\end{cases}
\quad \mbox{and} \quad
\delta_j^B = \begin{cases}
	D & w(j+1) =  B \ \mbox{or}\ j = 2n+2\\
	U & \mbox{else}.
\end{cases}
\]
Note by construction that $\delta^B_1 = U$.
For $\delta^A_i = D$ and $\delta^B_j = D$ ($j \neq 2n+2$), define labels $\lambda^A$ and $\lambda^B$ by
\[
\lambda^A(i) = 1 + \# \{k < i: \delta^A_k = D, \pi^{-1}_k > \pi^{-1}_i\} \quad \mbox{and} \quad \lambda^B(j) = 1 + \#\{\ell > j: \delta^B_\ell = D, \pi_{j-1} > \pi_{\ell-1}\}.
\]
For example, with $\tau = 9 4 3 2 8 \ 10\ 7 5 1 6$, we have $w = A A B B A A A B B B$ and
\[
(\delta^A,\lambda^A) = 
\begin{tikzpicture}[scale=.4]
	\node[circle, fill = black, scale = 0.3] at (0,0) {};
	\node[circle, fill = black, scale = 0.3] at (1,1) {};
	\node[circle, fill = black, scale = 0.3] at (2,2) {};
	\node[circle, fill = black, scale = 0.3] at (3,3) {};
	\node[circle, fill = black, scale = 0.3] at (4,2) {};
	\node[circle, fill = black, scale = 0.3] at (5,3) {};
	\node[circle, fill = black, scale = 0.3] at (6,4) {};
	\node[circle, fill = black, scale = 0.3] at (7,3) {};
	\node[circle, fill = black, scale = 0.3] at (8,2) {};
	\node[circle, fill = black, scale = 0.3] at (9,1) {};
	\node[circle, fill = black, scale = 0.3] at (10,0) {};
	
	\draw (0,0) -- (1,1) -- (2,2) -- (3,3) -- (4,2) -- (5,3) -- (6,4) -- (7,3) -- (8,2) -- (9,1) -- (10,0);
	
	\node at (3.75,3) {2};
	\node at (6.75,4) {3};
	\node at (7.75,3) {2};
	\node at (8.75,2) {1};
	\node at (9.75,1) {1};
	
	\node[red] at (3.75,5) {2};
	\node[red] at (6.75,5) {7};
	\node[red] at (7.75,5) {5};
	\node[red] at (8.75,5) {1};
	\node[red] at (9.75,5) {6};
	
\end{tikzpicture},
\quad
(\delta^B,\lambda^B) = 
\begin{tikzpicture}[scale=.4]
	\node[circle, fill = black, scale = 0.3] at (0,0) {};
	\node[circle, fill = black, scale = 0.3] at (1,1) {};
	\node[circle, fill = black, scale = 0.3] at (2,2) {};
	\node[circle, fill = black, scale = 0.3] at (3,3) {};
	\node[circle, fill = black, scale = 0.3] at (4,2) {};
	\node[circle, fill = black, scale = 0.3] at (5,1) {};
	\node[circle, fill = black, scale = 0.3] at (6,2) {};
	\node[circle, fill = black, scale = 0.3] at (7,3) {};
	\node[circle, fill = black, scale = 0.3] at (8,4) {};
	\node[circle, fill = black, scale = 0.3] at (9,3) {};
	\node[circle, fill = black, scale = 0.3] at (10,2) {};
	\node[circle, fill = black, scale = 0.3] at (11,1) {};
	\node[circle, fill = black, scale = 0.3] at (12,0) {};
	
	\draw[dashed] (0,0) -- (1,1);
	\draw (1,1) -- (2,2) -- (3,3) -- (4,2) -- (5,1) -- (6,2) -- (7,3) -- (8,4) -- (9,3) -- (10,2) -- (11,1);
	\draw[dashed] (11,1) -- (12,0);
	
	\node at (3.75,3) {3};
	\node at (4.75,2) {2};
	\node at (8.75,4) {2};
	\node at (9.75,3) {1};
	\node at (10.75,2) {1};
	
	\node[blue] at (3.75,5) {3};
	\node[blue] at (4.75,5) {2};
	\node[blue] at (8.75,5) {5};
	\node[blue] at (9.75,5) {1};
	\node[blue] at (10.75,5) {6};
\end{tikzpicture}.
\]
Here, the red integers indicate $\sigma^{-1}_i$ for $\delta^A_i = D$ and the blue integers indicate $\sigma_{j-1}$ for $\delta^B_j = D$.

Let $(\hat{\delta}^B,\lambda^B)$ be the lattice path obtained by ignoring the first and last step of $\delta^B$.
\begin{proposition}
	\label{p:ethiopian}
	For $\pi \in S_{2n}$, $(\delta^A(\pi),\lambda^A) = (\hat{\delta}^B(\pi),\lambda^B)$ if and only if $\pi \in \Ifpf_{2n}$.
\end{proposition}

\begin{proof}
	Let $w$ be the allocation function for $\pi$, and write 
	\[
	w^{-1}(A) = \{a_1<\dots<a_n\}, \quad w^{-1}(B) = \{b_1 < \dots < b_n\}.
	\]
	By construction $\delta^A = \hat{\delta}^B$ if and only if both $\pi(w^{-1}(B) = w^{-1}(A))$, hence $\pi^{-1}(w^{-1}(B)) = w^{-1}(A)$.
	Moreover, $\lambda^A = \lambda^B$ if and only if $\pi(b_i) < \pi(b_j)$ for $i < j$.
	This is equivalent to saying $\pi(a_i) = b_i$ and $\pi(b_i) = a_i$, that is $\pi \in \Ifpf_{2n}$.
\end{proof}

We have the surprising but easy consequence that for $\Ifpf_{2n}$ Biane's bijection coincides with theirs.
\begin{corollary}
	For $\tau \in \Ifpf_{2n}$, $\phi(\tau) = (\delta^A(\tau),\lambda^A) = (\hat{\delta}^B(\tau),\lambda^B)$.
\end{corollary}

Call a permutation $\pi$ \emph{fair} if Alice and Bob eat the same morsels regardless of who chooses first.
For a fixed-point-free involution, Proposition~\ref{p:ethiopian} implies whenever Alice picks the morsel with value $i$ that Bob will next pick the morsel at position $i$.
This guarantees any $\tau \in \Ifpf_{2n}$ is fair.

\begin{proposition}
\label{p:ethiopian_order}
A permutation $\pi \in \mathfrak{S}_{2n}$ is fair if and only if $\hat{\delta}^B(\pi)$ is a labeled Dyck path.
\end{proposition}

\begin{proof}
Let $a_i$ and $b_i$ be the position of the $i$-th morsel eaten by Alice and Bob respectively.
These values can be recovered from $\hat{\delta}^B(\pi)$.
Based on the optimal strategy, Alice ate her morsels starting from the right, so $a_i$ is the position of the $i$-th rightmost up step in $\hat{\delta}^B(\pi)$.
To recover Bob's moves, start from the leftmost down step and move right.
If a down step is labelled $k$, match this down step with the $k$-th leftmost up step that has not been matched yet.
Then $b_i$ is the position of the down step matched with the up step at $a_i$.
Note that $\hat{\delta}^B(\pi)$ is a labeled Dyck path if and only if $a_i$ is to the left of $b_i$ for all $i \in [n]$.
 
Let $a_i'$ and $b_i'$ be the position of the $i$-th morsel eaten by Alice and Bob respectively when Bob started first.
Similar to the Alice-first variant, the optimal strategy is found starting with the last moves.
Suppose that $a_n', b_n', \hdots, a_{i+1}', b_{i+1}'$ are equal to their counterparts of the original game.
Then $a_{i}' = \min \{a_1, b_1, \hdots, a_{i}, b_{i}\} = \min \{a_{i}, b_{i}\}$.
Therefore if $a_i' = a_i$ then $a_{i} < b_{i}$.
Now, suppose that $a_n, b_n, \hdots, b_{i+1}, a_{i}$ are equal to their counterparts.
Then $b_i'$ is Alice's least favorite morsel among $\{b_1, a_1, \hdots, a_{i-1}, b_i\}$, but $b_i$ is known to be Alice's least favorite morsel in the larger set $\{a_1, b_1, \hdots, b_i, a_i\}$, so $b_i' = b_i$.
Thus by induction $\pi$ is fair if and only if $a_i < b_i$ for all $i \in [n]$, which is equivalent to $\hat{\delta}^B(\pi)$ being a labelled Dyck path.
\end{proof}

\begin{remark}
As a consequence of the crossout procedure in~\cite{billera2015decompose}, Alice will always eat morsels from right to left and Bob will always eat morsels from highest to lowest.
Therefore, for a fair permutation Alice and Bob will eat the same morsels in the same order regardless of who chooses first.
\end{remark}

\begin{corollary}
	\label{c:ethiopian_count}
The number of fair permutations of length $2n$ is $(2n-1)!!^2.$
\end{corollary}

\begin{proof}

By Proposition ~\ref{p:ethiopian_order} and~\cite[Thm.~1]{billera2015decompose}, counting fair permutations is equivalent to counting pairs of labelled Dyck paths of length $n$.
Setting $q=1$ in Corollary ~\ref{c:watson}, this is $(2n-1)!!^2$. 
\end{proof}

Call a permutation \emph{$k$-fair} if Alice eats all but $k$ of the same morsels when going second.
Note the fair permutations are precisely the 0-fair permutation.
It would be interesting to study enumerative properties of $k$-fair permutations, and structural properties of their corresponding Dyck paths.

\subsection{A final identity} \label{ss:deodhar-srinivasan}
The following equation is~\cite[Thm. 1.2]{deodhar2001statistic} and its translation to Motzkin paths:
\begin{equation}
	\label{eq:deodhar}
	\binom{n}{k}_q = \sum_{\sigma \in \cI_n} (q-1)^{c(\sigma)} q^{\ellfpf(\sigma)} \binom{n - 2c(\sigma)}{k - c(\sigma)} = \sum_{\mu \in \cM_n} (q-1)^{s(\mu)} \tilde{H}[\mu;q] \binom{n-2s(\mu)}{k-s(\mu)}
\end{equation}
where $s(\mu)$ is the number of down steps in $\mu$.
The proof of Equation~\eqref{eq:deodhar} follows from generating function manipulations, and it would be interesting to give a direct combinatorial proof using Motzkin paths.

\bibliographystyle{amsalpha} 
\bibliography{references}

\end{document}